\documentclass[12pt]{amsart}
\numberwithin{equation}{section}
\usepackage{amsmath,amsthm,amsfonts,amscd,eucal,mathabx}
\usepackage{xypic}
\usepackage{graphicx}
\usepackage{amssymb}
\def\cb{{\mathcal B}}

\def\ch{{\mathcal H}}

\def\cp{{\mathcal P}}

\def\cs{{\mathcal S}}

\def\ga{{\mathfrak A}}

\def\bc{{\mathbb C}}

\def\bn{{\mathbb N}}

\def\bq{{\mathbb Q}}
\def\br{{\mathbb R}}
\def\bt{{\mathbb T}}

\def\bz{{\mathbb Z}}

\def\a{\alpha}
\def\b{\beta}

\def\th{\theta} 
\def\om{\omega}

\newtheorem{thm}{Theorem}[section]
\newtheorem{lem}[thm]{Lemma}

\newtheorem{cor}[thm]{Corollary}
\newtheorem{prop}[thm]{Proposition}
\newtheorem{rem}[thm]{Remark}
\newtheorem{defin}[thm]{Definition}

\theoremstyle{definition}
\newtheorem{examp}{Example}[section]

\DeclareMathAlphabet{\mathpzc}{OT1}{pzc}{m}{it}

\begin{document}

\title[Non-commutative skew-product extension dynamical systems]
{Non-commutative skew-product extension dynamical systems}

\author[V. Crismale]{Vitonofrio Crismale}
\address{Vitonofrio Crismale\\
Dipartimento di Matematica \\
Universit\`{a} degli Studi di Bari Aldo Moro \\
Via Edoardo Orabona 4, Bari 70125, Italy}\email{{\tt
vitonofrio.crismale@uniba.it}}

\author[S. Del Vecchio]{Simone Del Vecchio}
\address{Simone Del Vecchio\\
Dipartimento di Matematica \\
Universit\`{a} degli Studi di Bari Aldo Moro \\
Via Edoardo Orabona 4, Bari 70125, Italy} \email{{\tt
simone.delvecchio@uniba.it}}

\author[M.E.  Griseta]{Maria Elena Griseta}
\address{Maria Elena Griseta\\
Dipartimento di Matematica \\
Universit\`{a} degli Studi di Bari Aldo Moro \\
Via Edoardo Orabona 4, Bari 70125, Italy} \email{{\tt
mariaelena.griseta@uniba.it}}

\author[S. Rossi]{Stefano Rossi}
\address{Stefano Rossi\\
Dipartimento di Matematica \\
Universit\`{a} degli Studi di Bari Aldo Moro \\
Via Edoardo Orabona 4, Bari 70125, Italy} \email{{\tt
stefano.rossi@uniba.it}}


\begin{abstract}

Starting from a uniquely ergodic action of a locally compact group $G$ on a compact space $X_0$, we consider  
non-commutative skew-product extensions of the dynamics, on
the crossed product $C(X_0)\rtimes_\a\bz$, through a $1$-cocycle of $G$ in $\bt$, with $\a$ commuting with the given dynamics. We first prove that any such two skew-product extensions are conjugate if and only if
the corresponding cocycles are cohomologous. We then study unique ergodicity and unique ergodicity w.r.t. the fixed-point subalgebra by characterizing both in terms of the cocycle assigning the dynamics. The set of all invariant
states is also determined:  it is affinely homeomorphic
with $\cp(\bt)$, the Borel probability measures on the one-dimensional torus $\bt$, as long as the system is not uniquely ergodic. Finally, we show that unique ergodicity w.r.t. the fixed-point
subalgebra of a skew-product extension amounts to the uniqueness of an invariant conditional expectation onto the
fixed-point subalgebra

\end{abstract}

\maketitle

\section{Introduction}

First introduced in the 1950s by Anzai  \cite{A} and  thoroughly analyzed later in the 1960s by Furstenberg  \cite{Fu},  Anzai skew-products are a class of topological dynamical systems which is broad enough to produce a wide range of several different ergodic properties depending on the parameters assigning the system.\\
The starting data to define such systems is a uniquely ergodic system $(X_0, \theta)$, where $X_0$ is a compact Hausdorff space and $\theta$ a homeomorphism of $X_0$ with only one invariant measure $\mu_0$ (which is assumed fully supported).
If now $f$ is any continuous function from $X_0$ to $\bt$, the one-dimensional torus,  we can consider   the homeomorphism
\begin{equation*}
\Phi_f(x, z):= (\theta(x), f(x)z),\,\, (x,z)\in X_0\times\bt\,
\end{equation*}
acting on the product $X_0\times\bt$. The Anzai skew-product associated with $f$ is by definition the dynamical system given by the pair $(X_0\times\bt, \Phi_f)$. The product measure $\mu_0\times m$, with $m$ being the normalized Lebesgue measure on $\bt$, is invariant for $\Phi_f$ irrespective of what $f$ is.  More interestingly, the ergodicity $\mu_0\times m$ 
is the same as requiring that  $\mu_0\times m$ is  the only $\Phi_f$-invariant measure. 
Furthermore, the ergodicity of the product measure can be  recast more concretely in terms of the following equations in the unknown function $g$ in $L^\infty(X_0, \mu_0)$
\begin{equation*}
g(\theta(x)) f^n(x)=g(x)\,\,\, \mu_0-\textrm{almost everywhere}\,\,
\end{equation*}
known as the cohomological equations. Precisely, the product measure is ergodic ({\it i.e.} extreme in the convex set of all invariant measures) if and only if for every $n\neq 0$ the cohomological equations only have the trivial solution $g=0$. It is a deep, highly non-trivial result that the non-existence of non-trivial continuous solutions is equivalent to the minimality  ({\it i.e.} there are no proper closed invariant subsets) of the system instead. 
This yields an interesting, much cited example of a minimal system that nevertheless fails to be uniquely ergodic, see \cite[Corollary 12.8]{K} or \cite[Theorem 7.1]{M}, or \cite{R}.
In recent years non-commutative versions of the Anzai skew-products recalled above have been obtained in the framework
of $C^*$-algebras and their automorphisms. The earliest of these versions was first considered on the non-commutative two torus in \cite{OP}, and later studied in \cite{DFGR} with a focus on ergodic theory aspects. It was realized soon after that most of the results on the Anzai skew-products on the two torus carry over to the more general non-commutative setting of crossed products, \cite{ DFR1}.
 In particular,  the ergodicity of a distinguished invariant state, which will be recalled later on in the paper,  is ruled by the behaviour of certain cohomological equations, whose unknowns are now operators in a von Neumann algebra, which is commutative as long as the crossed product is obtained out of a commutative $C^*$-algebra.  Moreover, 
the ergodicity of this state still amounts to the unique ergodicity of the system. 
Allowing more generality, one is led to consider unique ergodicity with respect to the fixed-point subalgebra as well. This property was first introduced by Abadie and Dykema \cite{AD} as a generalization of
unique ergodicity where one requires that every state on the fixed-point subalgebra has only one invariant extension to the whole algebra.  In \cite{DFR1}  it was shown that for a skew-product unique ergodicity w.r.t. the fixed-point subalgebra implies continuity of all solutions to the cohomological equations. The converse was proved to hold only for classical systems, leaving out the case of a general crossed product. One of our original motivations to write the present paper was to bridge this gap.  Another was to answer in our setting a question   posed about a decade ago by
Abadie and Dykema, \cite[Question 3.4]{AD},  whether the uniqueness of an invariant conditional expectation onto the fixed-point subalgebra implies unique ergodicity w.r.t. the fixed-point subalgebra.
It turns out that this is the case for skew products, as we prove in Theorem \ref{AbDyk}. This result is
worth comparing with the general picture, where the question is in fact answered in the negative, as 
shown by a counterexample due to D. Ursu, Example \ref{Ursu}.\\
Having these goals in mind, we have taken this opportunity to further extend our scope by letting a locally compact group $G$ (satisfying a number of mild assumptions), rather than a single automorphism, act on our crossed products. Not only does this generality allow us to consider a wider class of $C^*$-dynamical systems, but, and possibly more importantly, it lays greater emphasis on the interpretation of the cohomological equations as coboundary equations for those  $1$-cocycles on $G$ that the dynamics is assigned through.
For a comprehensive account in the classical setting of skew-product extension by the action of a locally compact group the reader is referred to \cite{KS, K}. \\
 That said, we can now move on to present our results while describing how the paper is organized as well. 
The preliminary section acquaints the reader with most of the prerequisites needed to give self-contained proofs of the paper's results.\\
Section \ref{prel} provides a rather quick account of what we deem is strictly necessary to know about the crossed product of a $C^*$-algebra by a single automorphism.
Section \ref{generg} contains a detailed treatment of unique ergodicity w.r.t. the fixed-point subalgebra in our general context of non-commutative
$C^*$-algebras acted upon by a unimodular, amenable, second countable locally compact group $G$.  In more detail, Proposition \ref{ADgen} is the characterization
of unique ergodicity w.r.t. the fixed point subalgebra in terms of the pointwise norm convergence of the
Ces\`{a}ro  averages, whereas Proposition \ref{vonNeumannG} is a general version of von Neumann's ergodic theorem, where the convergence of the 
Ces\`{a}ro averages holds at the Hilbert space level. Although these results might be more or less known or at least easily guessed, their proofs are not entirely
trivial, nor are their statements easily found in the literature of the field. Therefore, we have decided to include them in the paper all the same in the hope they may also serve
as a quick yet readable introduction to the topic.\\
In Section 3, starting from  a uniquely ergodic system $(\ga, G, \theta)$, with unique $G$-invariant state $\om_0$, and
a $1$-cocycle $(u_g)_{g\in G}$  of $G$ in $\bt$, we define the extended skew product
as the action of $G$ on  the crossed product $\ga\rtimes_\a\bz$ through the $*$-automorphisms $\Phi_g^u$ given by
\begin{equation*}
\Phi_g^u(a):=\theta_g(a), \textrm{for all}\,a\in\ga, \quad \Phi_g^u(V):=u_g V\, .
\end{equation*}
where $V$ is the unitary implementator of $\a$, that is $VaV^*=\a(a)$, for all $a$ in $\ga$.
The action of $\Phi_g^u$ on the powers of $V$ is given by $\Phi_g^u(V^n)= u_g^{(n)}V^n$, where 
 $(u_g^{(n)})_{g\in G}$ is for every $n$ in $\bz$ a new  $1$-cocycle of $G$ which takes account of $\a$.\\
In analogy with the classical case, the state $\om$ on  $\ga\rtimes_\a\bz$ given by 
\begin{equation*}
\om(aV^n)=\om_0(a)\delta_{n, 0}\,,\,\, a\in\ga,\,n\in\bz
\end{equation*}
 is a distinguished invariant state, which we shall refer to as the canonical invariant state.\\
In Section \ref{classify}, we show that two such  systems are equivalent, either as topological or measurable systems, 
if and only if the corresponding $1$-cocycles are cohomologous, with a continuous or a measurable coboundary, Theorem \ref{classification}.\\
In Section \ref{dichotomy}, we analyze the fixed-point subalgebra by proving that a strong dichotomy occurs:
it is either trivial or isomorphic with $C(\bt)$, the continuous functions on $\bt$, Proposition \ref{dicho}. Moreover, 
for the fixed-point subalgebra to be trivial it is necessary and sufficient that none of the cocycles $\{(u_g^{(n)})_{g\in G}\colon n\neq 0\}$ is a continuous coboundary, Proposition \ref{weakergo}. Section \ref{dichotomy} ends
with Theorem \ref{uniquelyerg}, where unique ergodicity w.r.t. the fixed-point subalgebra is shown to hold if and only if for every $n$ in $\bz$ the $1$-cocycle $(u_g^{(n)})_{g\in G}$ is not a measurable non-continuous coboundary, thus solving the problem left open in \cite{DFR1}.\\
Section \ref{invstate} is devoted to the study of the convex set of all invariant states of the extended skew-product.
Once again, a strong dichotomy takes place, for this set is either a singleton or affinely homeomorphic with $\cp(\bt)$, the set of all Borel probability measures on $\bt$, Theorem \ref{invstates}. Phrased differently, as soon as
the system fails to be uniquely ergodic, namely some of the $1$-cocycles $(u_g^{(n)})_{g\in G}$ are (possibly measurable) coboundaries, its invariant states feature a rich structure, which is morally reminiscent
of the fixed-point subalgebra even at the level of the von Neumann algebra generated in the GNS representation
of the canonical invariant state. Finally, Theorem \ref{AbDyk} completes the characterization of unique ergodicity w.r.t. the fixed-point subalgebra in terms of invariant conditional expectations onto the fixed-point subalgebra.
In particular, this result extends \cite[Theorem 4.2]{DFR2} showing that actually all non-commutative skew-products are well behaved in relation to Abadie and Dykema's problem.\\
As an outlook for future research, we would like to pose a question: is it true that  a non-commutative skew product is minimal in the sense of Longo and Peligrad, \cite[Definition 2.1]{LP}, if and only if none of the cocycles $\{(u_g^{(n)})_{g\in G}: n\neq 0\}$ is a continuous coboundary? An affirmative answer to this question might yield an example of a non-trivial non-commutative $C^*$-dynamical system that is minimal without being uniquely ergodic. 

\section{Preliminaries}

\subsection{Crossed products}\label{prel}
Throughout the paper, we will make extensive use of the crossed product of a given $C^*$-algebra by the action of a single automorphism.
Therefore, to keep the exposition as self-contained as we possibly can, we start by recalling what that crossed product is.
Given a  (unital) $C^*$-algebra $\ga$ and an automorphism $\a$ of $\ga$, the  universal (or maximal) crossed product of $\ga$ by $\a$ is the
universal $C^*$-algebra generated by $\ga$ and a unitary $V$ such that $VaV^*=\alpha(a)$ for all $a$ in $\ga$.
Such a $C^*$-algebra can be shown to exist, see {\it e.g.} \cite{D}, and it is denoted by  $\ga\rtimes_\a\bz$.\\
By its very definition, the unitary $V$ is actually defined up to a phase, meaning that the one-dimensional torus, $\bt$, acts naturally on  $\ga\rtimes_\a\bz$
through the so-called gauge automorphisms $\{\a_z: z\in\bt\}$, determined by $\a_z(a)=a$ for all $a$
 in $\ga$  and $\a_z(V)=zV$. In particular, it follows that the spectrum of $V$ is the whole $\bt$.\\
The $C^*$-subalgebra $\ga\subset\ga\rtimes_\a\bz$ is easily seen to be the invariant subalgebra under the gauge action.
Furthermore, by averaging the gauge action one obtains a canonical (faithful) conditional expectation, $E$, from  $\ga\rtimes_\a\bz$ onto $\ga$.
Explicitly, $E$ is defined as $E(x):=\int_\bt \a_z(x){\rm d}m(z)$, $x$ in $\ga\rtimes_\a\bz$, where $m$ is the normalized-Lebesgue measure of $\bt$. Note
that $E$ is completely determined by $E(aV^n)=a\delta_{n,0}$, for all $a$ in $\ga$ and $n$ in $\bz$, where $\delta$ is the Kronecker symbol.\\
The canonical conditional expectation allows for the Fourier expansion in the setting of crossed products. 
Indeed, it is known that any $x$ in $\ga\rtimes_\a\bz$ can be written as
$x=\sum_{k\in\bz} V^kE(V^{-k}x)$, where the convergence of the series is in norm but in the  Ces\`{a}ro sense, \cite[Theorem VIII.2.2]{D}.\\

In the sequel we will need to identify $C^*(\ga, V^n)$ with $\ga\rtimes_{\a^n}\bz$ for a fixed integer $n$. 
This is possible thanks to the following proposition, whose proof
is included in full detail for want of a reference.

\begin{prop}\label{injectivity}
For every $n$ in $\bz$, the $C^*$-subalgebra $C^*(\ga, V^n)\subset \ga\rtimes_\a\bz$ is canonically
isomorphic with $\ga\rtimes_{\a^n}\bz$.
\end{prop}

\begin{proof}
Thanks to  the universal property enjoyed by $\ga\rtimes_{\a^n}\bz$, there exists a surjective $*$-homomorphism
$\Psi\colon\ga\rtimes_{\a^n}\bz\rightarrow C^*(\ga, V^n)$ such that $\Psi(a)=a$, for all $a$ in $\ga$, and $\Psi(W)=V^n$, where
$W$ in $\ga\rtimes_{\a^n}\bz$ is the implementator of $\a^n$.\\
We need to show that $\Psi$ is injective. Let $x$ in  $\ga\rtimes_{\a^n}\bz$ be such that $\Psi(x)=0$.
Considering the Fourier expansion of $x$ allows us to rewrite it as $x=\sum_{k\in\bz}  W^ka_k$, for some sequence $\{a_k: k\in\bz\}\subset\ga$.
Correspondingly $\Psi(x)$, too, can be expanded as $\Psi(x)=\sum_{k\in\bz} V^{nk}a_k$. Now $\Psi(x)=0$ is only possible if $a_k=0$ for all $k$ in $\bz$ by virtue of
the uniqueness of the Fourier expansion in $\ga\rtimes_\a\bz$. But then $x=0$ as well, and we are done. 
\end{proof}

\subsection{General results on $C^*$-dynamical systems}\label{generg}

By a $C^*$-dynamical system we mean a triple $(\ga, G, \a)$, where $\ga$ is a (unital) $C^*$-algebra, $G$ a locally compact group, and
$\a: G\rightarrow {\rm Aut}(\ga)$ a group homomorphism. In other words, associated with every $g$ in $G$ there is an automorphism
$\a_g$ of $\ga$ such that $\a_{gh}=\a_g\circ \a_h$ for all $g, h$ in $G$. Furthermore, as is commonly assumed, we also need to require the following continuity condition:
for every $a$ in $\ga$, the map $G \ni g\mapsto \a_g(a)\in\ga$ is continuous as a function from $G$ with its topology to $\ga$ endowed with its norm topology.\\
Given a $C^*$-dynamical system $(\ga, G, \a)$, a state $\om$ on $\ga$ is said to be invariant under the dynamics, or more shortly $G$-invariant, if
$\om\circ\a_g=\om$ for all $g$ in $G$. The (compact convex) set of all $G$-invariant states is denoted by $\cs(\ga)^G$.\\
A $C^*$-dynamical system is said to be {\it uniquely ergodic} if $\cs(\ga)^G$ is a singleton. Phrased differently, unique ergodicity means
that the system only has one invariant state. Unique ergodicity is the strongest notion of ergodicity one may consider. For instance, under suitable assumptions on the group $G$, {\it e.g.} amenability, it implies
that the fixed-point subalgebra, $\ga^G:=\{a\in\ga: \a_g(a)=a, g\in G\}$,  is trivial.\\
An interesting generalization of unique ergodicity was introduced by Abadie and Dykema in \cite{AD}, where they consider the action of a single automorphism, 
that is $G=\bz$, by allowing
the fixed-point subalgebra to be non-trivial. Precisely, a $C^*$-dynamical system $(\ga, \bz, \a)$ is {\it uniquely ergodic w.r.t. the fixed-point subalgebra} if
any state on $\ga^\bz$ admits exactly one $\bz$-invariant extension to the whole $C^*$-algebra $\ga$. This property can be characterized in a number of ways,
the most relevant of which is that it is equivalent to the pointwise convergence in norm of the  Ces\`{a}ro averages, see \cite[Theorem 3.2]{AD}.  As is noted in that paper, the convergence in norm of the  Ces\`{a}ro averages has the important consequence that there exists a unique invariant conditional expectation onto the fixed-point subalgebra, which is of course
obtained by taking the limit of the  Ces\`{a}ro averages themselves.  In the same paper, Abadie and Dykema raised
the question as to whether the uniqueness of such an invariant conditional expectation is enough
for unique ergodicity w.r.t. the fixed-point algebra to hold. The answer in the negative to that problem has recently been found by D. Ursu, who exhibited an example of a classical dynamical system which fails to be
uniquely ergodic w.r.t. the fixed point subalgebra while featuring a unique conditional expectation.
Since this example has not been published and is rather part of a private communication, we would like to include it below.
\begin{examp}\label{Ursu}
Let $X = \bz\bigcup\{-\infty\}\bigcup\{\infty\}$ be the two-point compactification of $\bz$. 
Let $\a$ be the “shift by 1” homeomorphism. Note that $f\in C(X)$ with $f\circ\a=f$ implies $f$ constant.\\
Let $Y$ be the space that is obtained by glueing together $[0, 1]$ and $X$ identifying $1$ in the former space and $-\infty$ in the latter. We may extend $\a$ to $Y$ by acting trivially on $[0, 1]$. We continue to denote this extension
by $\a$.
Now $C[0, 1]$ embeds unitally into $C(Y )$ by extending any of its functions $f$  to take the constant value $f(1) =f(-\infty)$ on all of $X$. Under this embedding, it is clear that $C(Y )^\a$ = $C[0, 1]$.\\
We verify that there is a unique conditional expectation $E\colon C(Y )\rightarrow  C(Y)^\a\cong C[0, 1]$. 
Let $f\in C(Y )$ and let be $z$ in $[0, 1)$ be a fixed point. 
Now let $h$ be any function in $C(Y)$ such such that $h(z)=1$ with $h$ vanishing on $X$. Note that
$h$ sits in $C(Y)^\a\cong C[0, 1]$. We have:
\begin{align*}
E(f)(z) = E(f)(z) · h(z) = E(fh)(z) = (fh)(z) = f(z) · h(z) = f(z)\,.
\end{align*}
Since this is true for any $z \in [0, 1)$, by continuity, $E(f)(z) = f(z)$ for $z = 1$ in $[0, 1]$ as well. 
In other words, the only conditional expectation is the “restriction” map onto [0, 1]. This is clearly $\a$-invariant as well.
However, we observe that the two states $\delta_{-\infty}$ and $\delta_\infty$ are two different $\a$-invariant extensions of the state $\delta_1$ on $C[0, 1]$.
\end{examp}

In order to deal with our skew-product dynamical systems, we need to generalize Abadie and Dykema's characterization of unique ergodicity w.r.t. the fixed-point subalgebra so as to cover a wider class of groups.
Precisely, henceforth we will be working under the assumption that $G$ is a unimodular, amenable, second-countable locally compact group.
We will denote by $\mu_G$ its (bi-invariant) Haar measure.
The hypotheses made on the group guarantee that $G$ posseses a two-sided F{\o}lner sequence $\{F_n: n\in\bn\}$, see \cite[Section 1.1, Proposition 2]{OW}.
This means that there exists a sequence of compact subsets $F_n\subset G$, $n$ in $\bn$, such that for every $h$ in $G$ one has
$\lim_{n\rightarrow\infty} \frac{\mu_G(F_n\Delta hF_n)}{\mu_G(F_n)}=0$ and $\lim_{n\rightarrow\infty} \frac{\mu_G(F_n\Delta F_nh)}{\mu_G(F_n)}=0$,
where $\Delta$ denotes the symmetric difference between sets.\\
Given a $C^*$-dynamical system $(\ga, G, \a)$, for every $a$ in $\ga$, we define the sequence
$\{M_n(a): n\in\bn\}$ of its Ces\`{a}ro averages as
\begin{equation*}\label{Cesaromean}
M_n(a):=\frac{1}{\mu_G(F_n)}\int_{ F_n}\a_g(a){\rm d}\mu_G(g)\,,
\end{equation*}
where $\{F_n: n\in\bn\}$ is any (two-sided) F{\o}lner sequence. The integral is
understood in the sense of Bochner and is well defined thanks to the continuity of the function $G \ni g\mapsto \a_g(a)\in\ga$, see {\it e.g.}
\cite[Theorem 1, Section 5, Chapter V]{Y} .\\
Note that, for each $n$ in $\bn$, $\ga\ni a\mapsto M_n(a)\in\ga$ defines a bounded linear map from $\ga$ to itself, with 
$\|M_n\|\leq 1$ for all $n$.\\
That said, we are ready to state the announced generalization of Abadie and Dykema's characterization of unique ergodicity with respect to 
the fixed-point subalgebra in the general framework of group actions on non-commutative $C^*$-algebras.

\begin{prop}\label{ADgen}
For a $C^*$-dynamical system $(G, \a, \ga)$ with $G$ as above, the following are equivalent:
\begin{itemize} 
\item [(i)] Every state on $\ga^G$ has a unique $G$-invariant extension to $\ga$;
\item [(ii)] For every $a$ in $\ga$, $M_n(a)$ converges in norm to some (invariant) element in $\ga$.
\end{itemize}
\end{prop}

\begin{proof}
The implication (ii) $\Rightarrow$ (i) immediately follows from  the equality $\om(a)=\om(M_n(a))$, $a$ in $\ga$ and $n$ in $\bn$, which holds for any invariant state $\om$ and which shows that $\om$ is uniquely determined by its restriction
to $\ga^G$.\\
As for the reverse implication, we start by observing that (i) implies that any bounded linear functional
on $\ga^G$ has a unique bounded $\a$-invariant extension to the whole $\ga$, {\it cf.} \cite[Observation 3.1]{AD}.
We aim to  prove that the linear subspace $Y:=\ga^G+ {\rm span}\{a-\a_h(a): a\in\ga, h\in G\}$ is norm dense
in $\ga$. To this end,  by the Hahn-Banach theorem it is enough to show that if a bounded linear functional 
vanishes on $Y$, then it is zero on $\ga$. If $\varphi$ is any such functional, then
$\varphi$ is zero on $\ga^G$ and  is $\a$-invariant as well. But then $\varphi$ must be zero on $\ga$ by the unique extension property.\\
We next show that (ii) holds for $a=b-\a_h(b)$, with $b$ in $\ga$, for a fixed $h$ in $G$. 
We do this by showing that for such elements $a$ the corresponding sequence $\{M_n(a): n\in\bn\}$ is Cauchy:
\begin{align*}
&\|M_n(b-\a_h(b))- M_{n+k}(b-\a_h(b))\|\\
&=\left\|\frac{1}{\mu_G(F_n)}\int_{ F_n}(\a_g(b-\a_h(b)){\rm d}\mu_G(g)- \frac{1}{\mu_G(F_{n+k})}\int_{ F_{n+k}}(\a_g(b-\a_h(b)){\rm d}\mu_G(g)\right\|\\
&\leq\left(\frac{\mu_G(F_n\Delta F_nh)}{\mu_G(F_n)}+\frac{\mu_G(F_{n+k}\Delta F_{n+k}h)}{\mu_G(F_{n+k})}\right)\|b\|\leq \varepsilon
\end{align*}
for all $k$ in $\bn$ as soon as $n\geq N_\varepsilon$, where  $N_\varepsilon$ is such that 
$\frac{\mu_G(F_l\Delta F_lh)}{\mu_G(F_l)}\leq \frac{\varepsilon}{2\|b\|}$ for all $l\geq N_\varepsilon$, and such
$N_\varepsilon$ exists by definition of a F{\o}lner sequence.\\
By linearity of the average operators $M_n$, condition (ii) holds for all $a$ in $Y$.\\
The full conclusion is reached thanks to the density of $Y$ in $\ga$ by an $\frac{\varepsilon}{3}$-argument.
Let $a$  be in $\ga$ and let $b$ be in $Y$ with $\|a-b\|\leq\frac{\varepsilon}{3}$. Then we have:
\begin{align*}
&\|M_n(a) - M_{n+k}(a)\|\\
&\leq\|M_n(a) - M_n(b)\|+\|M_n(b) - M_{n+k}(b)\|+\|M_{n+k}(b) - M_{n+k}(a)\|\\
&\leq\frac{2\varepsilon}{3}+\|M_n(b) - M_{n+k}(b)\|\leq \varepsilon\,,
\end{align*}
for all $k$ in $\bn$ and $n\geq N_\varepsilon$, where  $N_\varepsilon$ in $\bn$ is such that
$n\geq N_\varepsilon$ implies  $\|M_n(b) - M_{n+k}(b)\|\leq \frac{\varepsilon}{3}$.
The proof is thus complete.
\end{proof}
At this point the characterization of unique ergodicity, which is a more or less known fact, {\it cf.} \cite[Proposition 2.4]{LP}, is an easy application of Proposition \ref{ADgen}
\begin{cor}
A $C^*$-dynamical system $(G, \a, \ga)$ with $G$ as above is uniquely ergodic with unique $G$-invariant state 
$\om$ if and only if for all $a$ in $\ga$  one has $\lim_{n\rightarrow\infty }M_n(a)=\om(a)1$ in norm.
\end{cor}

\begin{proof}
It is enough to show that unique ergodicity implies $\ga^G=\bc$. If $\ga^G$ is not trivial, then 
there exist at least two different states $\om_1, \om_2$ on it. Denote by $\widetilde{\om}_i$ is any  extension to
the whole $\ga$ of $\om_i$, $i=1, 2$.
Consider the two sequences of states $\left\{\frac{1}{\mu_G(F_n)}\int_{ F_n}\widetilde{\om}_i\circ\a_g \,{\rm d}\mu_G(g): n\in\bn\right\}$, $i=1, 2$. 
By compactness, for $i=1, 2$, there exists  $\varphi_i$, which is an  accumulation point of $\left\{\frac{1}{\mu_G(F_n)}\int_{ F_n}\widetilde{\om}_i\circ\a_g \,{\rm d}\mu_G(g): n\in\bn\right\}$. By construction
$\varphi_1$ and $\varphi_2$
are both $G$-invariant, and yet they do not agree on
$\ga^G$, hence $\varphi_1\neq \varphi_2$.
\end{proof}

The last proposition of this section is a von Neumann ergodic theorem tailored to our framework. 
Although it is likely to be a known result, we do include
a complete proof all the same, for want of a reference. 
Before stating it, let us also establish some notation.\\
Given a $G$-invariant state $\om$, for all $g$ in $G$, we denote by $V_g^\om$ the unitary implementator of $\a_g$ on the GNS Hilbert space
$\ch_\om$, that is
\begin{equation}\label{unitaryimpl}
V_g^\om\pi_\om(a)\xi_\om:=\pi_\om(\a_g(a))\xi_\om,\,\,\textrm{for all}\,\, a \in \ga\,.
\end{equation}
 By definition $V_g^\om\pi_\om(a)(V_g^\om)^*=\pi_\om(\a_g(a))$, for all $a$ in $\ga$ and
$g$ in $G$. We denote by $\ch_\om^G$ the closed subspace of all $G$-invariant vectors of $\ch_\om$, that is $\ch_\om^G:=\{x\in\ch_\om: V_g^\om x=x, \textrm{for all}\, g\in G\}$ and by $E_\om$ the orthogonal projection of $\ch_\om$ onto $\ch_\om^G$.\\
Now, as is easy to verify, the continuity of $G \ni g\mapsto \a_g(a)\in\ga$ (where $a$ is any fixed element of $\ga$) implies that, for any
invariant state $\om$ on $\ga$, the map $G\ni g\mapsto V_g^\om\in
\cb(\ch_\om)$ is a strongly continuous unitary representation of $G$, namely, for any vector $x$ in $\ch_\om$, the function 
$G\ni g\mapsto V_g^\om x\in\ch_\om$ is norm continuous. We are now ready to state our version of von Neumann's ergodic theorem.

\begin{prop}\label{vonNeumannG}
Let $(\ga, G, \a)$ be a $C^*$-dynamical system with $G$ as above and let $\om$ be in $\cs^G(\ga)$. Then for any $x$ in $\ch_\om$ one has
$$\lim_{n\rightarrow\infty}  \frac{1}{\mu_G(F_n)}\int_{F_n} V_g^\om x\,{\rm d}\mu_G(g)= E_\om x\,,$$
where the convergence holds in norm and $\{F_n: n\in\bn\}$ is any  F{\o}lner sequence for $G$.
\end{prop}

\begin{proof}
If $x$ sits in $\ch_\om^G$, then  $\frac{1}{\mu_G(F_n)}\int_{F_n} V_g^\om x\,{\rm d}\mu_G(g)=x$ for all $n$, and there is nothing to prove.
The thesis will follow if we show that for every $x$ in $(\ch_\om^G)^\perp$ one has 
$$\lim_{n\rightarrow\infty}  \frac{1}{\mu_G(F_n)}\int_{F_n} V_g^\om x\,{\rm d}\mu_G(g)= 0\, .$$
To this end, note that $(\ch_\om^G)^\perp=\overline{span}\{{\rm Ran}(I-V_h^\om): h\in G\}$.  Thus linearity and a standard
approximation argument enable us to consider $x$ of the form $x=(I-V_h^\om)y$ for some $h$ in $G$ and $y$ in $\ch_\om$. For such an 
$x$ we have:
\begin{align*}
 &\frac{1}{\mu_G(F_n)}\int_{F_n} V_g^\om x\,{\rm d}\mu_G(g)= \frac{1}{\mu_G(F_n)}\int_{F_n} V_g^\om (I-V_h^\om)y\,{\rm d}\mu_G(g)\\
&=\frac{1}{\mu_G(F_n)}\int_{F_n} V_g^\om y\,{\rm d}\mu_G(g)-\frac{1}{\mu_G(F_n)}\int_{F_n} V_{gh}^\om y\,{\rm d}\mu_G(g)\\
&=\frac{1}{\mu_G(F_n)}\int_{F_n} V_g^\om y\,{\rm d}\mu_G(g)-\frac{1}{\mu_G(F_n)}\int_{F_nh} V_g^\om y\,{\rm d}\mu_G(g)\\
&=\frac{1}{\mu_G(F_n)}\int_{F_n\setminus F_nh} V_g^\om y\,{\rm d}\mu_G(g)\,
\end{align*}
hence
$$\left\|\frac{1}{\mu_G(F_n)}\int_{F_n} V_g^\om x\,{\rm d}\mu_G(g)\right\|\leq \frac{\mu_G(F_n\Delta F_nh)}{\mu_G(F_n)}\|y\|\rightarrow 0$$
as $n$ goes to $\infty$ thanks to the definition of a F{\o}lner sequence.
\end{proof}
\begin{rem}
The above result continues to hold under the milder assumption that $G$ is only amenable.
\end{rem}

\section{Non-commutative skew-product extensions}
We start with a $C^*$-dynamical system $(\ga, G, \theta)$, where $\ga=C(X_0)$ (with $X_0$ being a metrizable compact space) is a commutative $C^*$-algebra, $G$ a unimodular, amenable, second-countable locally compact
group, and $\theta: G\rightarrow {\rm Aut}(\ga)$ a group homomorphism.
From now on we will always work under the assumption that $(\ga, G, \theta)$ is uniquely ergodic with unique invariant state
$\om_0$ ( $\om_0(f)=\int_{X_0} f{\rm d}\mu_0$, $f$ in $C(X_0)$, for a unique probability Borel measure $\mu_0$).

By a continuous $1$-cocycle on $G$ we will always mean a continuous function 
$u: G\times X_0\rightarrow\bt$ such that
\begin{equation}\label{cocycle1}
u(gh, x)=u(h, \theta_g(x))u(g, x)\,, \,\textrm{for all}\, g, h\,\in G, x\in X_0
\end{equation}
For any fixed $g$ in $G$, we can think of the function $X\ni x\mapsto u(g, x)\in \bt$ as
a unitary element of the $C^*$-algebra $\ga=C(X_0)$, which we will henceforth denote by $u_g$. In terms of
$u_g$   \eqref{cocycle1} rewrites as
\begin{equation}\label{cocycle2}
u_{gh}=\theta_g(u_h)u_g\,, \,\textrm{for all}\, g, h\,\in G\,.
\end{equation}
Note that if $u=(u_g)_{g\in G}$ and $v=(v_g)_{g\in G}$ are $1$-cocycles, then
$uv^*=(u_gv_g^*)_{g\in G}$ is still a $1$-cocycle.

For a given automorphism (homeomorphism) $\a$ of $\ga=C(X_0)$ we consider the universal crossed product
$\ga\rtimes_\a\bz$.  As of now, we will assume that $\a$ commutes with the action of $G$, that is
$\a\circ\theta_g=\theta_g\circ \a$, for all $g$ in $G$. Let $u=(u_g)_{g\in G}$ be a $1$-cocycle
w.r.t. the action $\theta$. For every $g$ in $G$, the universal property of the crossed product allows us to define a
unique automorphism $\Phi_g^u$ of $\ga\rtimes_\a\bz$, whose action on the generators is 
\begin{equation*}\label{Anzaidef}
\Phi_g^u(a):=\theta_g(a),\quad \textrm{for all}\,a\in\ga, \quad \Phi_g^u(V):=u_g V\, .
\end{equation*}
Furthermore, thanks to \eqref{cocycle2} one readily sees that $\Phi_{gh}^u=\Phi_g^u\circ \Phi_h^u$. 
We will denote the $C^*$-dynamical system
thus obtained by $(\ga\rtimes_\a\bz, G, \Phi^u)$.
Compounding $\om_0$ with the canonical conditional expectation, $E$, onto $\ga$ yields a distinguished $G$-invariant state $\om:=\om_0\circ E$. The $G$-invariance of $\om$ follows at once from the equality
$\om(aV^k)=\om_0(a)\delta_{k, 0}$, for all $a$ in $\ga$ and $k$ in $\bz$.\\
Finally, for every $g$ in $G$ we denote by $\widetilde{\Phi}^u_g$ the action of $\Phi^u_g$ at the level of von 
Neumann algebra $\pi_\om(\ga\rtimes_\a\bz)''$, namely $\widetilde{\Phi}^u_g(T)=V_g^\om T (V_g^\om)^* $, for all
$T$ in  $\pi_\om(\ga\rtimes_\a\bz)''$, where  $V_g^\om$ is the unitary implementator of $\Phi_g^u$ which was defined in
Equation \eqref{unitaryimpl}.
\begin{examp}
The general construction above applies in particular to the non-commutative (irrational) two torus $\mathbb{A}_\a$ (the universal $C^*$-algebra generated by two unitaries $U$ and $V$ subject to the commutation rule
$UV= e^{2\pi i\a}VU$) as soon as this is seen as the crossed product of $C(\bt)$ by the action of the homeomorphism $\a$ given by the rotation
$\a(z)=e^{2\pi i\a}z$, $z$ in $\bt$, where with a slight abuse of notation $\a$ denotes both the homeomorphism and the rotation angle ($\frac{\a}{2\pi}$ lies in $\br\setminus\bq$). \\
We take $G=\bz$, which means we need only define an automorphism acting on $\mathbb{A}_\a$.
The automorphism $\theta$ on $C(\bt)$ we start from is the one corresponding to the rotation $\theta(z)=e^{2\pi i\theta}z$, $z$ in $\bt$, where again $\frac{\theta}{2\pi}$ is assumed irrational. Therefore, the dynamical system $(C(\bt), \theta)$ is uniquely ergodic with its unique invariant measure given by the normalized Lebesgue measure $m$ on $\bt$. Now, since $\bz$ is a singly generated group, defining a $1$-cocycle on $\bz$ amounts to choosing a unitary $u$ in $C(\bt)$. Indeed, if we denote by $\Phi^u_1$ the automorphism on the non-commutative two torus
determined by $\Phi^u_1(U)=e^{2\pi i\theta} U$ and $\Phi^u_1(V)=uV$, then, for each integer $n$, $\Phi^u_n:= (\Phi^u_1)^n$ acts on $V$
\begin{equation*}
\Phi^u_n(V)=\left\{\begin{array}{ll}
                     V\,,\,\,\, n=0\,,&\\[1ex]
                     u\theta(u)\ldots \theta^{n-1}(u) V\,,   \,\,\,
                      n>0\,,& \\[1ex]
                   \theta^{n}(u^*)\theta^{n+1}(u^*)\ldots\theta^{-1}(u^*)V\,,  \,\,\,
                      n<0\, &
                    \end{array}
                    \right.
\end{equation*}
which allows us to identify the $1$-cocycle $(u_n)_{n\in\bz}$ on $\bz$ at once.
\end{examp}

\subsection{Classification in terms of cocycles}\label{classify}

The present section is devoted to the classification of our skew-product systems up to a natural notion of conjugacy. We actually consider two such notions depending on what we lay the emphasis on: topological versus measurable dynamics. We start by giving a couple of definitions.

\begin{defin}
A $1$-cocycle $u=(u_g)_{g\in G}$ is a continuous (or $L^\infty$) coboundary if there exists a continuous (or in $L^\infty(X_0,\mu_0)$)
function (known as $0$-cocycle) $w\colon X\rightarrow \bt$ such that
\begin{equation*}\label{coboundary}
u_g=\theta_g(w^*)w\,, \,\, \textrm{for all}\,\, g\in G\,.
\end{equation*}
Two $1$-cocycles $u=(u_g)_{g\in G}$ and $v=(v_g)_{g\in G}$ are $C^*$ (or $W^*$) cohomologous if the $1$-cocylcle
$uv^*=(u_gv_g^*)_{g\in G}$  is a continuous (or measurable) coboundary.
\end{defin}

\begin{defin}
We say that two  skew products  $(\ga\rtimes_\a\bz, G,  \Phi^u)$ and $(\ga\rtimes_\a\bz, G, \Phi^v)$  are 
$C^*$-conjugate if there exists a $*$-isomorphism $\Psi$ of $\ga\rtimes_\a\bz$ such that
$\Psi(a)=a$, for all $a$ in $\ga$, and $\Phi^u_g\circ\Psi=\Psi\circ \Phi^v_g$ for all $g$ in $G$.\\
Analogously,  $(\ga\rtimes_\a\bz, G,  \Phi^u)$ and $(\ga\rtimes_\a\bz, G, \Phi^v)$  are 
$W^*$-conjugate if there exists a  $*$-isomorphism $\Psi$ of $\pi_\om(\ga\rtimes_\a\bz)''$ such that
$\Psi(\pi_\om(a))=\pi_\om(a)$, for all $a$ in $\ga$, and $\widetilde{\Phi}^u_g\circ\Psi=\Psi\circ \widetilde{ \Phi}^v_g$ for all $g$ in $G$.

\end{defin}

We are now ready to state the result on the classification of our systems.

\begin{thm}\label{classification}
Two  skew products  $(\ga\rtimes_\a\bz, G,  \Phi^u)$ and $(\ga\rtimes_\a\bz, G, \Phi^v)$ 
are $C^*$ (or $W^*$) conjugate if and only if the corresponding $1$-cocycles $u=(u_g)_{g\in G}$, 
 $v=(v_g)_{g\in G}$ are  $C^*$ (or $W^*$) cohomologous.
\end{thm}

\begin{proof}
We start by dealing with the case when one of the two systems is trivial, say $v=(v_g)_{g\in G}=1$, namely
$v_g=1$ for all $g$ in $G$.
If $v=1$,  then $u=(u_g)_{g\in G}$ is a coboundary, that is there exists a unitary function
$w: X\rightarrow \bt$ satisfying $u_g=\theta_g(w^*)w$ for all $g$  in $G$.\\
Suppose first that $w$ is a continuous function.
By universality of the crossed product, the map $\Psi(a):=a$, $a$ in $\ga$, $\Psi(V):=wV$ extends to
an endomorphism of $\ga\rtimes_\a\bz$ which is seen at once to be an automorphism. We next show that the equality
$\Psi\circ\Phi_g^1= \Phi_g^u\circ\Psi$ holds for all $g$ in $G$.
To this end, it is enough to check the equality on $V$ (as it obviously holds for all $a$ in $\ga$), which is done below
\begin{align*}
\Psi\circ\Phi_g^1(V)=wV=\theta_g(w)u_gV=\Phi_g^u\circ\Psi(V)\, .
\end{align*}
Conversely if $(\ga\rtimes_\a\bz, G, \Phi^u)$ and $(\ga\rtimes_\a\bz, G, \Phi^1)$ are $C^*$-conjugate through
an automorphism $\Psi$ such that
$\Psi\circ\Phi_g^1= \Phi_g^u\circ\Psi$, then
$\Psi(V)=wV$ for some $w$ in $\ga$ thanks to Lemma \ref{autoconj}.
The same computation as above shows that imposing the intertwining relation $\Psi\circ\Phi_g^1= \Phi_g^u\circ\Psi$ on
$V$ implies that $w$ satisfies $u_g=\theta_g(w^*)w$ for all $g$  in $G$, hence $u$ is $C^*$-coboundary.\\
Suppose now that the the solution $w$ is only measurable, that is $w$ sits in  $\pi_{\om_0}(\ga)''\cong L^\infty(X, \mu_0)$
and satisfies $\widetilde{\th}_g(w)u=w$ for all $g$ in $G$. Our next aim is to show that there exists an automorphism $\Psi$ of the von Neumann algebra
$\pi_\om(\ga\rtimes_\a\bz)''$ such that 
$\Psi(\pi_\om(a))=\pi_\om(a)$, $a$ in $\ga$, and $\Psi(\pi_\om(V))=w\pi_\om(V)$. To begin with, by universality of the crossed product there does exist
an injective $*$-homomorphism $\Psi_o$ from $\pi_\om(\ga\rtimes_\a\bz)$ to $\pi_\om(\ga\rtimes_\a\bz)''$ such that
$\Psi_o(\pi_\om(a))=\pi_\om(a)$, $a$ in $\ga$, and $\Psi_o(\pi_\om(V))=w\pi_\om(V)$. Now denote by $\widetilde{\om}$ the vector state of $\cb(\ch_\om)$ associated
with $\xi_\om$. We next show that the equality $\widetilde{\om}\circ\Psi_o=\widetilde{\om}$ holds on $\pi_\om(\ga\rtimes_\a\bz)$. Indeed, for all
$a$ in $\ga$ and $k$ in $\bz$, we have
\begin{align*}
\widetilde{\om}\circ\Psi_o\, (\pi_\om(aV^k))=\widetilde{\om}(\pi_\om(a) (w\pi_\om(V))^k)=\delta_{k, 0}\widetilde{\om}(\pi_\om(a))=\widetilde{\om}(\pi_\om(aV^k))\,.
\end{align*}
This allows us to define an isometry $U_o$ on the dense subspace $\pi_\om(\ga\rtimes_\a\bz)\xi_\om\subset \ch_\om$ as
$$U_o \left(\sum_k \pi_\om(a_kV^k)\xi_\om\right):=\Psi_o\left(\sum_k \pi_\om(a_kV^k)\right)\xi_\om\,.$$
We denote by $U$ the extension of $U_o$ to the whole Hilbert space $\ch_\om$ and note that $U$ is actually
a unitary whose inverse can be obtained by performing the above construction again with $w$ replaced by
$w^*$.\\
By construction the restriction of ${\rm ad}(U)$ to $\pi_\om(\ga\rtimes_\a\bz)$ coincides with $\Psi_o$. Therefore,
the von Neumann algebra $\pi_\om(\ga\rtimes_\a\bz)''$ is invariant under ${\rm ad}(U)$ as well, and the restriction of
${\rm ad}(U)$ to  $\pi_\om(\ga\rtimes_\a\bz)''$ provides an injective homomorphism $\Psi$ of $\pi_\om(\ga\rtimes_\a\bz)''$
into itself, which is surjective  because $\Psi(w^*V)=V$. Finally, the intertwining relation
$\Psi\circ\widetilde{\Phi}_g^1=\widetilde{\Phi}_g^ u\circ\Psi$, for all $g$ in $G$,  can be verified exactly as above.\\
The implication  "$W^*$-conjugacy $\Rightarrow$  $u=(u_g)_{g\in G}$ is a $W^*$-coboundary" can be proved exactly as in the case of $C^*$-conjugacy by applying Lemma \ref{autoconj} ($W^*$-case).\\

The general case can be reconducted to the foregoing analysis in the following way. We now know that the $1$-cocycles $u=(u_g)_{g\in g}$ and
 $v=(v_g)_{g\in G}$ are  $C^*$ (or $W^*$) cohomologous if and only if
 $(\ga\rtimes_\a\bz, G,  \Phi^1)$ and $(\ga\rtimes_\a\bz, G, \Phi^{uv^*})$ are $C^*$ (or $W^*$) conjugate, that is if there exists an automorphism $\Psi$ of $\ga\rtimes_\a\bz$ (or $\pi_\om(\ga\rtimes_\a\bz)''$) such that
$\Psi\circ\Phi_g^1=\Phi_g^{uv^*}\circ\Psi$ (or $\Psi\circ\widetilde{\Phi}_g^1=\widetilde{\Phi}_g^{uv^*}\circ\Psi$) for all $g$ in $G$.
Now $\Psi(V)=wV$ (or $\Psi(\pi_\om(V))=w\pi_\om(V)$) with $w$ in $\ga$ (or $w$ in $\pi_{\om}(\ga)''$) satisfying $\theta_g(w)uv^*=w$ (or $\widetilde{\theta}_g(w)uv^*=w$) for all $g$ in $G$. The conclusion is got to by verifying that $\Psi$ also intertwines $\Phi_g^u$ and $\Phi_g^v$ for all $g$ in $G$, {\it i.e.} 
$\Psi\circ\Phi_g^v=\Phi_g^u\circ\Psi$, for all $g$ in $G$,  As usual, this equality need only be checked on $V$:
\begin{align*}
\Psi\circ\Phi_g^v\,(V)&=\Psi(v_gV)=v_gwV=v_g\theta_g(w)u_gv_g^*V\\
&=\theta_g(w)u_gV=\Phi_g^u\circ\Psi\, (V)\,,
\end{align*}
and we are done. The counterpart in the $W^*$ setting is entirely analogous.
\end{proof}

\subsection{Weak ergodicity and ergodicity w.r.t. the fixed-point subalgebra}\label{dichotomy}
This section addresses the study of the main ergodic properties a skew-product system may enjoy. In particular, 
weak ergodicity, unique ergodicity and unique ergodicity w.r.t the fixed-point subalgebra can all
be characterized in terms of coboundaries.\\ To state our results, we first need to recall that
given a $1$-cocycle $u=(u_g)_{g\in G}$ of $G$, it is possible to define a family of $1$-cocycles  as
\begin{equation*}\label{cocycleVn}
u_g^{(n)}=\left\{\begin{array}{ll}
                     I\,,\,\,\, n=0\,,&\\[1ex]
                     u\a(u_g)\ldots \a^{n-1}(u_g)\,,   \,\,\,
                      n>0\,,& \\[1ex]
                   \a^{n}(u_g^*)\a^{n+1}(u_g^*)\ldots\a^{-1}(u_g^*)\,,  \,\,\,
                      n<0\,.&
                    \end{array}
                    \right.
\end{equation*}
Note that, for any fixed $g$ in $G$, $(u_g^{(n)})_{n\in\bz}$ is a $\bz$ $1$-cocycle w.r.t the action of $\a$. Furthermore, by a straightforward induction on
$n$ we see that
\begin{equation*}\label{actionVn}
\Phi_g^u(V^n)= u_g^{(n)}V^n\,, g\in G, n\in\bz\, .
\end{equation*}
As we show below, for weak ergodicity to occur it is necessary and sufficient that the so-called cohomological equations at the $C^*$-algebra level do not have non-trivial solutions.
\begin{prop}\label{weakergo}
The $C^*$-dynamical system  $(\ga\rtimes_\a\bz, G,  \Phi^u)$ is weakly ergodic if and only the $1$-cocycle $(u_g^{(n)})_{g\in G}$ is not a continuous coboundary for every
$n$ in $\bz$.\\
In other words,  $(\ga\rtimes_\a\bz, G,  \Phi^u)$ is weakly ergodic if and only if for every $n\in\bz$ with $n\neq 0$ the equation in the unknown
$a$ in $\ga$ 
$$\theta_g(a)u_g^{(n)}=a\, ,\,\,{\textrm for\,\, all}\, g\in G$$
has only the trivial solution $a=0$.
\end{prop}

\begin{proof}
Recall that any $x$ in the crossed product $\ga\rtimes_\a\bz$ can be expanded as
$x=\sum_n \a^{-n} (E(V^{-n}x))V^n$,  where the convergence of the series is understood in the norm topology in the Cesàro sense.
This allows us to study the fixed-point equation $\Phi_g^u(x)=x$, $g$ in $G$, for $x$ of the form $aV^n$, for some $a$ in $\ga$ and
$n$ in $\bz$. For such an $x$, the fixed-point equation $\Phi_g^u(aV^n)=aV^n$, $g$ in $G$, becomes $\theta_g(a)u_g^{(n)}V^n=aV^n$, t
 which is satisfied if and only if $\theta_g(a)u_g^{(n)}=a$.\\
Now if $a\in\ga$ is a non-trivial solution of $\theta_g(a)u_g^{(n)}=a$, for all $g\in G$, then $w:=\frac{|a|}{\|a\|}$ 
 satisfies $\theta_g(w)=w$ for all $g$ in $G$, hence $w=1$ by ergodicity. This shows that if  $\theta_g(a)u_g^{(n)}=a$ has non-trivial
solution, then there is always a unitary solution, {\it i.e.} $(u_g^{(n)})$ is a continuous coboundary.\\
Finally, the conclusion is arrived at
by noting that for $n=0$ the equation is $\theta_g(a)=a$ for all $g$ in $G$, which is satisfied only by multiples of the identity again thanks to
ergodicity of $(\ga, G, \theta)$. 
\end{proof}
In the sequel, the equations $\theta_g(a)u_g^{(n)}=a$ will be referred to as the $C^*$-cohomological equations and their solutions will be called 
continuous solutions. \\

Define $\bz_u:=\{n\in\bz: (u_g^{(n)})\,\,\textrm{is a continuous coboundary}\}$
\begin{prop}\label{Zsubgroup}
$\bz_u$ is a subgroup of $\bz$, hence $\bz_u=m_o\bz$ for some (possibly zero) natural $m_0$. \\
Moreover, for every $n$ in $\bz_u$ the continuous $0$-cocycle that determines the coboundary $(u_g^{(n)})_{g\in G}$ is unique up to a phase.
\end{prop}
\begin{proof}
The group property follows exactly as in \cite[Proposition 10.2]{DFR1}. The second part of the statement amounts to showing
that the set of all solutions of  $\theta_g(a)u_g^{(n)}=a$, $g$ in $G$, is of the form $\{\lambda w_n: \lambda\in\bc\}$,
where $w_n$ is a unitary solution of the equation, and this follows again by unique ergodicity of $(\ga, G, \theta)$.
\end{proof}
At this point, we are in a position to prove that the fixed-point subalgebra is always commutative irrespective of the cocycle $(u_g)_{g\in G}$. More precisely, the fixed-point subalgebra is either trivial or isomorphic with $C(\bt)$.
\begin{prop}\label{dicho}
If the $C^*$-dynamical system  $(\ga\rtimes_\a\bz, G,  \Phi^u)$ is not weakly ergodic, then the fixed-point subalgebra
$(\ga\rtimes_\a\bz)^G$ is isomorphic with $C(\bt)$.
\end{prop}
\begin{proof}
By Propositions \eqref{weakergo} and \eqref{Zsubgroup}, the fixed-point subalgebra coincides with the subalgebra generated by the unitary
$w_{m_0}V^{m_0}$, where $w_{m_0}$ is the unique (up to phase) unitary such that  $\theta_g(w_{m_0})u_g^{(m_0)}=w_{m_0}$.
To this aim, take the $C^*$-subalgebra of $\ga\rtimes_\a\bz$ generated by $\ga$ and $V^{m_0}$. By
Proposition \ref{injectivity} this is isomorphic with the universal crossed product $\ga\rtimes_{\a_{m_0}}\bz$. The universal property
of the crossed product allows us to consider a $*$-endomorphism $\Psi$ of $C^*(a, V^{m_0})$   uniquely
determined by $\Psi(a)=a$, for all $a$ in $\ga$, and $\Psi(V^{m_0})=w_{m_0}V^{m_0}$. Now $\Psi$ is a $*$-isomorphism with inverse
$\Psi^{-1}(a)=a$, for all $a$ in $\ga$, and $\Psi^{-1}(V^{m_0})=w_{m_0}^*V^{m_0}$. The conclusion is now reached by observing that 
the restriction of $\Psi$ to $C^*(V^{m_0})\cong C(\bt)$ realizes a $*$-isomorphism between  $C^*(V^{m_0})$ and $C^*(w_{m_0}V^{m_0})$.
\end{proof}

As we did in Section \ref{generg}, we denote by $V_g^\om$  the unitary that implements $\Phi_g^u$ on the GNS representation of the crossed product
$\ga\rtimes_\a\bz$ of the state $\om$.\\ 
The Hilbert space $\ch_\om$ of the GNS representation of our distinguished state $\om$ was shown in \cite[Lemma 8.3]{DFR1} to decompose into a direct sum of the type  $\mathcal{H}_\omega=\oplus_{n\in\mathbb{Z}} \mathcal{H}_n$,
with $\ch_0:=[\pi_\om(\ga)\xi_\om]\cong\ch_{\om_o}$ and $\ch_n:=\pi_\om(V)^n[\pi_\om(\ga)\xi_\om]$ for all $n\neq 0$.
Finally we denote by $V_g^{\om_0}$ the unitary on $\ch_{\om_0}$ implementing $\theta_g$ for all $g$ in $G$. Note that
$V_g^{\om_0}$ is unitarily equivalent to the restriction of $V_g^{\om}$ to $\ch_0$.\\
The above decomposition enables us to study the fixed-point equation $V_g^{\om}x=x$, $x$ in $\ch_\om$, at the Hilbert space level. 
Any scalar multiple of the GNS vector $\xi_\om$ is obviously fixed by $V_g^\om$ for all $g$ in $G$. We shall refer to multiples of $\xi_\om$ as the trivial
fixed vectors of $V_g^\om$. 
\begin{prop}
The unitary group $\{V_g^\om: g\in G\}$ has non-trivial fixed vectors in $\ch_\om$ if and only 
if for some $n\neq 0$ the cocycle $(u_g^{(n)})_{g\in G}$ is a measurable coboundary. \\
\end{prop}
\begin{proof}
All we need to show is that there are non-trivial solutions  $x\in\ch_\om$ of $V_g^\om x=x$, for all $g$ in $G$, if and only if there exists 
$n\neq 0$ such that the equation
\begin{equation}\label{ncoboundary}
\eta(\theta_g(x))u_g^{(n)}(x)=\eta(x)\,, \mu_0\,\, a.e.\,,\,\,g\in G
\end{equation}
has a non-zero solution $\eta$ in $\ch_{\om_0}=L^2(X, \mu_0)$.\\
A straightforward adaptation of the proof of  \cite[8.5]{DFR1} shows that there exist non-trivial $x$ in $\ch_\om$ with $V_g^\om x =x$, for all $g$ in $G$, if and only if 
\begin{equation}\label{comH}
\pi_{\om_0}(\a^{-n}(u_g^{(n)}))V_g^{\om_0}\eta=\eta\,,\,\,g\in G
\end{equation}
has non-zero solutions $\eta$ in $\ch_{\om_0}=L^2(X_0, \mu_0)$ for some $n\neq 0$. But with $\ga=C(X_0)$ and $\om_0$ given by
$\om_0(f)=\int_{X_0} f{\rm d}\mu_0$, $f$ in $C(X_0)$, Equation \eqref{comH} reads as
\begin{equation}\label{comL2}
u_g^{(n)}(\a^{-n}(x))\eta(\theta_g(x))=\eta(x)\,,\,\, \mu_0\, a.e.\,,\, \textrm{for all} \,g\,\,\in G
\end{equation}
 The conclusion is reached by observing that if $\eta$ in $L^2(X_0, \mu_0)$ is a non-trivial solution of
\eqref{comL2}, then $\eta'(x)=\eta(\a^n(x))$, $\mu_0$ a.e, satisfies \eqref{ncoboundary}.
\end{proof}
Here follows the characterization of unique ergodicity w.r.t. the fixed-point subalgebra.
\begin{thm}\label{uniquelyerg}
A necessary and sufficient condition for $(\ga\rtimes_\a\bz, G, \Phi^u)$ to be uniquely ergodic w.r.t. the
fixed-point subalgebra is that for every $n$ in $\bz$ the $1$-cocycle $(u_g^{(n)})_{g\in G}$ is not a measurable non-continuous coboundary.
\end{thm}

\begin{proof}
We start by proving that the condition is necessary. We suppose that $(\ga\rtimes_\a\bz, G, \Phi^u)$ is uniquely ergodic
w.r.t. the fixed-point subalgebra and show, for every $n$ in $\bz$, that if the $1$-cocycle $(u_g^{(n)})_{g\in G}$ is a coboundary, then it is
a continuous coboundary.
Suppose there exists some $n$ in $\bz$ such that $(u_g^{(n)})_{g\in G}$ is a measurable coboundary. Then there exists
$\eta$ in $L^\infty(X_0, \mu_0)$  (thought of as $\pi_{\om}(C(X_0))''$ acting on $\ch_\om$) such that $\eta\pi_\om(V^n)\xi_\om$ is in
$\ch_\om^G$. Now if $a$ is any  element of $\ga$, by Proposition \ref{ADgen} the sequence $\{M_k(aV^n): k\in \bn\}$ converges in norm
to an element of the form $bV^n$, for some $b$ in $\ga$. We claim that $\pi_\om(bV^n)\xi_\om$ lies in
$\ch_\om^G$. But then the uniqueness up to a multiplicative scalar of the solutions to the cohomological equations at a fixed level implies
$\eta=\lambda\pi_\om(b)$, that is $\eta$ is a continuous solution. Let us now move on to show that the claim holds true. 
For every $h$ in $G$ we have:
\begin{align*}
V_h^\om \pi_\om(bV^n)\xi_\om &=\lim_{k\rightarrow\infty} V_h^\om\pi_\om(M_k(aV^n))\xi_\om\\
&=\lim_{k\rightarrow\infty} V_h^\om\pi_\om\left(\frac{1}{\mu(F_k)}	\int_{F_k}\a_g(aV^n){\rm d}\mu_G(g)\right)\xi_\om\\
&=\lim_{k\rightarrow\infty} \pi_\om\left(\frac{1}{\mu(F_k)}	\int_{F_k}\a_{hg}(aV^n){\rm d}\mu_G(g)\right)\xi_\om\\
&=\lim_{k\rightarrow\infty} \pi_\om\left(\frac{1}{\mu(F_k)}	\int_{hF_k}\a_g(aV^n){\rm d}\mu_G(g)\right)\xi_\om\\
&=\pi_\om(\lim_{k\rightarrow\infty} M_k(aV^n))\xi_\om=\pi_\om(bV^n)\xi_\om\,,
\end{align*}
where the second-last equality  follows from
$$\lim_{k\rightarrow\infty}\frac{1}{\mu(F_k)}	\int_{hF_k}\a_g(aV^n){\rm d}\mu_G(g)=\lim_{k\rightarrow\infty}M_k(a)\,, a\in\ga\,, $$
which in turn is a consequence of the F{\o}lner property. Indeed, we have:
\begin{align*}
&\left\|\int_{hF_k}\a_g(aV^n){\rm d}\mu_G(g)-M_k(a) \right\|\\
&=\left\|\int_{hF_k}\a_g(aV^n){\rm d}\mu_G(g)-\int_{F_k}\a_g(aV^n){\rm d}\mu_G(g) \right\|\\
&\leq \frac{\mu(hF_k\Delta F_k)}{\mu(F_k)}\|a\|\rightarrow 0
\end{align*}
as $k$ goes to $\infty$, which proves the claim.

We next prove that the condition is sufficient as well. To this aim, we will make use of a characterization of
unique ergodicity w.r.t. the fixed-point subalgebra due to Abadie and Dykema, \cite[Theorem ]{AD}: unique ergodicity is the same as requiring that every state defined on the fixed-point subalgebra uniquely extends to an invariant state on the whole
algebra.\\
This is in turn equivalent to showing that every  $\Phi^u$-invariant state $\varphi$
on $\ga\rtimes_\a\bz$ is completely determined by its restriction to $(\ga\rtimes_\a\bz)^G$.
Let $m_0$ be the minimum of the strictly positive integers $n$ such that there exists a non-null
$a\in \ga$ such that $\a^{-m_0}(u_g^{m_0})\th(a)=a$, for all $g$ in $G$.\\
The first thing we do is to compute the value of $\varphi$ on a monomial of the type $aV^{km_0}$. To this end, we first note that
\begin{align*}
&\lim_n M_n(aV^{km_0})=\lim_n M_n (aw_{km_0}^*w_{km_0}V^{km_0})\\
&=\lim_n\ m_n(aw_{km_0}^*)w_{km_0}V^{km_0} = \om_0(aw_{km_0}^*)w_{km_0}V^{km_0}
\end{align*}
At this point $\varphi(aV^{km_0})$ can be easily computed in the following way
\begin{align*}
\varphi(aV^{km_0})&=\varphi(aw_{km_0}^*w_{km_0}V^{km_0})=\lim_n\varphi\left(M_n(aw_{km_0}^*w_{km_0}V^{km_0})\right)\\
&=\om_0(aw_{km_0}^*w_{km_0})\varphi(w_{km_0}V^{km_0})\,,
\end{align*}
which shows that $\varphi(aV^{km_0})$ only depends on the restriction of $\varphi$ to $(\ga\rtimes_\a\bz)^G$.\\
We next show that if $l$ is not a multiple of $m_0$, then  $\varphi(aV^l)=0$ for all $a\in\ga$.
We shall argue by contradiction. Suppose there do exist $a$ in $\ga$ and $l$ in $\bz\setminus m_0\bz$ such that
$\varphi(aV^l)\neq 0$. We are going to exhibit a non-trivial solution of the cohomological equation right at the level
$l$. To this end, for all $g$ in $G$ and $l$ in $\bz$ we set $h_{g, l}:=\Phi_g^u(V^l)V^{-l}$. Note that
$h_{g, l}$ sits in $\ga$.\\
We start with the following preliminary limit equality

\begin{align*}
&\lim_n\pi_\varphi(M_n(aV^l))\xi_\varphi=\lim_n\frac{1}{|F_n|}\int_{F_n} U_g^\varphi    \pi_\varphi(aV^l) U_g^\varphi \xi_\varphi{\rm d}\mu_G(g)\\
&=\lim_n\frac{1}{|F_n|}\int_{F_n} U_g^\varphi    \pi_\varphi(aV^l)\xi_\varphi{\rm d}\mu_G(g)= E_\varphi\pi_\varphi(aV^l)\xi_\varphi\,,
\end{align*}
which holds in $\ch_\varphi$, with $E_\varphi$ being the orthogonal projection onto
$\ch_\varphi^G=\{\xi\in\ch_\varphi: U_g^\varphi\xi=\xi,\, \textrm{for all}\, g\in G\}$ thanks to
Proposition \ref{vonNeumannG}.\\
We  are going to show that $\pi_\varphi(V^{-l})E_\varphi\pi_\varphi(aV^l)\xi_\varphi$ provides the sought
solution. This can be seen in the following way
\begin{align*}
&\pi_\varphi(V^{-l})E_\varphi\pi_\varphi(aV^l)\xi_\varphi=\lim_n\pi_\varphi(V^{-l}M_n(aV^l))\xi_\varphi\\
&=\lim_n\frac{1}{|F_n|}\int_{F_n}\pi_\varphi(V^{-l}\theta_g(a) h_{g, l} V^l)\xi_\varphi{\rm d}\mu_G(g)\\
&=
\lim_n\frac{1}{|F_n|}\int_{F_n}\pi_\varphi(\alpha^l(\theta_g(a) h_{g, l}))\xi_\varphi{\rm d}\mu_G(g)\\
&=\lim_n\frac{1}{|F_n|}\int_{F_n}\om_0(\alpha^l(\theta_g(a) h_{g, l}))\xi_{\om_0}{\rm d}\mu_G(g)
\end{align*}
which shows that $\pi_\varphi(V^{-l})E_\varphi\pi_\varphi(aV^l)\xi_\varphi$ is in $\ch_{\om_0}$. Moreover

\begin{align*}
&\pi_{\om_o}(\a^{-n}(u_g^{(n)}))V_{\om_o,\th}\pi_\varphi(V^{-l})E_\varphi\pi_\varphi(aV^l)\xi_\varphi\, \\
&=\pi_{\varphi}(\a^{-n}(u_g^{(n)}))U_{\varphi}\pi_\varphi(V^{-l})E_\varphi\pi_\varphi(aV^l)\xi_\varphi\\
&=\pi_{\varphi}(\a^{-n}(u_g^{(n)})\a^{-n}(u_n^*)V^{-l})E_\varphi\pi_\varphi(aV^l)\xi_\varphi\\
&=\pi_{\varphi}(V^{-l})E_\varphi\pi_\varphi(a
V^l)\xi_\varphi
\end{align*}
which shows that $\pi_\varphi(V^{-l})E_\varphi\pi_\varphi(aV^l)\xi_\varphi$ is indeed a solution.
It only remains to show that $\pi_\varphi(V^{-l})E_\varphi\pi_\varphi(aV^l)\xi_\varphi\neq 0$ but this follows directly from
\begin{equation*}
\varphi(aV^\ell)=\langle\xi_\varphi,E_\varphi\pi_\varphi(aV^l) \xi_\varphi\rangle 
\end{equation*}
which in turn follows by the invariance of $\varphi$.
\end{proof}
From the proof of the foregoing result, we see at once that unique ergodicity can be characterized by
the non-existence of coboundaries ({\it i.e.} by the fact that the cohomological equations for $n\geq 1$ do not have non-trivial solutions at all).
\begin{cor}\label{uniqueerg}
A necessary and sufficient condition for $(\ga\rtimes_\a\bz, G, \Phi^u)$ to be uniquely ergodic, with unique
$G$-invariant state $\om$, is that for every $n$ in $\bz$ the $1$-cocycle $(u_g^{(n)})_{g\in G}$ is not a coboundary.
\end{cor}

\subsection{Invariant states and invariant conditional expectations}\label{invstate}
This section provides the complete description of the compact convex set of all invariant
states of a skew-product system. We will be working under the tacit assumption that the cohomological equations have non-trivial solutions, otherwise in the light of Corollary \ref{uniqueerg} the system will be uniquely ergodic and there is not much to say about the structure of the set of invariant states. Accordingly, $n_0$ will denote the smallest natural number for which the cohomological equations \eqref{ncoboundary} have a non-trivial solution, which we denote by $w_{n_0}$. Explicitly, $w_{n_0}$ is a function in $L^\infty(X_0, \mu_0)$ satisfying
$\theta_g(w_{n_0})u_g^{(n_0)}=w_{n_0}$  for all $g$ in $G$, the equality being understood in $L^\infty(X_0,  \mu_0)$.\\
The proof of the main result of the section requires some preliminary work to do. To this aim, we start by recalling how some natural conditional expectations,  which we will need to make use of, can be defined on a crossed product. We denote by $\beta_{n_0}$ the automorphism of $\ga\rtimes_\a\bz$ acting on the generators as
$\beta_{n_0}(a)=a$, for all $a$ in $\ga$, $\beta_{n_0}(V)=e^{\frac{2\pi i}{n_0}}V$. We then define
$(\ga\rtimes_\a\bz)^{\beta_{n_0}}:=\{x\in\ga\rtimes_\a\bz\colon \beta_{n_0}(x)=x\}$. Note that 
$(\ga\rtimes_\a\bz)^{\beta_{n_0}}$ is the $C^*$-subalgebra of  $\ga\rtimes_\a\bz$ generated by
$\ga$ and $V^{n_0}$.\\
We denote by $\mathcal{E}_{n_0}$ the canonical conditional expectation of $\ga\rtimes_\a\bz$ onto
$(\ga\rtimes_\a\bz)^{\beta_{n_0}}$, \emph{i.e. }$\mathcal{E}_{n_0}=\frac{1}{n_0}\sum_{k=0}^{n_0-1} \beta_{n_0}^k$.\\
Our next goal is to define a map $\widetilde{T}\colon(\ga\rtimes_\a\bz)^{\beta_{n_0}}\rightarrow \pi_\om((\ga\rtimes_\a\bz)^{\beta_{n_0}})''$ that acts trivially on $\ga$ and sends $V^{n_0}$ to
$w_{n_0}^*V^{n_0}$.
Note that this map is well defined, namely it extends to a $*$-homomorphism defined on the whole
$(\ga\rtimes_\a\bz)^{\beta_{n_0}}$ thanks to the universal property of the cross product, for the adjoint action of
$w_{n_0}^*V^{n_0}$ on $\ga$ is still  $\alpha$ due to $\ga$ being commutative.\\
Lastly, we would like to define an expectation $\Phi:\pi_\om((\ga\rtimes_\a\bz)^{\beta_{n_0}})''\rightarrow L^\infty(\bt, m)$ given by $\Phi(aV^{ln_0})=\om_0(a)V^{ln_0}$, for $a$ in $\ga$ and $l$ in $\bz$.
This is seen to extend on the whole $\pi_\om((\ga\rtimes_\a\bz)^{\beta_{n_0}})''$ by $\om$-invariance (which means
$\Phi$ is spatially implemented by an isometry of $\ch_\om$).
By composing the three maps above we get a map $T$ defined on $\ga\rtimes_\a\bz$ by
\begin{equation}\label{mapT}
T:=\Phi\circ\widetilde{T}\circ\mathcal{E}_{n_0}
\end{equation}
Note that
\begin{equation} \label{Tonmon}
T(aV^{kn_0})=\widetilde{\om}_0(aw_{k_0}^*)V^{kn_0}\,,\,\,\, \textrm{for all}\,\, a\, {\rm in}\, \ga,\, k\, {\rm in}\, \bz\,.
\end{equation}
The map $T$ is invariant under the dynamics and its range is as large as to contain all the powers of $V^{n_0}$.
\begin{lem}\label{RanT}
The unital completely positive map $T$ defined by \eqref{mapT} is invariant under  $\Phi^u_g$ for all $g$ in $G$.
Moreover, its range contains  ${\rm span}\{V^{kn_0}\colon k\in\bz\}$. 
\end{lem}
\begin{proof}
In order to prove the equality $T\circ\Phi^u_g=T$, $g$ in $G$,  it is enough to check it holds on monomials of the type
$aV^l$, $a$ in $\ga$ and $l$ in $\bz$. There are two cases depending on whether $l$ is or is not a multiple
of $n_0$. If $l$ is not a multiple of $n_0$, then the two sides of the equality certainly agree being both zero.
If $l=kn_0$ is a multiple of $n_0$, then  the right-hand side is $T(aV^{kn_0})=\widetilde{\om}_0(aw_{k_0}^*)V^{kn_0}$. As for the left-hand side, we have
\begin{align*}
T\circ\Phi^u_g(aV^{kn_0})&=T\circ\Phi^u_g(aw_{kn_0}^*w_{kn_0}V^{kn_0})= T(\widetilde{\theta}_g(aw_{kn_0}^*)w_{kn_0}V^{kn_0})\\
&=\widetilde{\om}_0(\widetilde{\theta}_g(aw_{kn_0}^*))V^{kn_0}=\widetilde{\om}_0(aw_{kn_0}^*)V^{kn_0}\, .
\end{align*}
As for the second part of the statement, set $\ga=C(X_0)$. We claim that for every $k$ in $\bz$ there exists $h$ (which may depend on $k$) in $C(X_0)$ with 
$\om_0(w_{kn_0}^*h)=\int \overline{u_{kn_0}}h {\rm d}\mu_0 \neq 0$ ($\mu_0$ is the probability measure on $X_0$ associated with the state $\om_0$). Since $T(hV^{kn_0})=\om_0(w_{kn_0}^*h)V^{kn_0}$ and $\om_0(w_{kn_0}^*h)$ is different from $0$, we see that $V^{kn_0}$ sits in the range of $T$.\\
All is left to do is prove the initial claim. To this end, consider a uniformly bounded sequence  
$\{h_l\colon l\in\bn\}\subset C(X_o)$ such that $h_l$ converges to $u_{kn_0}$  a.e. with respect to $\mu_o$. 
By Lebesgue's dominated convergence theorem we have $\lim_{l\rightarrow\infty} \int \overline{u_{kn_0}}h_l {\rm d}\mu =1$, which shows that for $l$ large enough $h_l$ will work as our $\widetilde{h}$.
\end{proof}

 In what follows $\overline{{\rm span}}\{V^{kn_0}\colon k\in\bz\}$ is identified with
$C(\bt)$ through the $*$-automorphism that sends the generator $V^{n_0}$ to the generator of
$C(\bt)$, that is the identical function $f(z)=z$, $z$ in $\bt$. 
This allows us to associate with any Borel probability measure $\mu$ on $\bt$ a state $\varphi_\mu$
 on $\overline{{\rm span}}\{V^{kn_0}\colon k\in\bz\}=C^*(V^{n_0})$, defined as
$\varphi_\mu(V^{kn_0}):=\int_\bt z^k {\rm d}{\mu}(z)$, $k$ in $\bz$.

In order to arrive at the sought characterization of the set of all invariant states, we also need a technical result
which basically says that the GNS representation of any invariant state contains a homomorphic image of the whole of the von Neumann algebra $L^\infty(X_0, \mu_0)$. More precisely, we have the following.
\begin{lem}
For any invariant state $\varphi$ in $\cs(\ga\rtimes_\a\bz)^G$, the GNS representation
$\pi_\varphi$ restricted to $\pi_{\om_0}(\ga)$  extends to a normal representation $\widetilde{\varphi}$ of $\pi_{\om_0}(\ga)''\cong L^\infty(X_0, \mu_0)$.
\end{lem}
\begin{proof}
Let $X$ be an operator belonging to the von Neumann algebra  $\pi_{\om_0}(\ga)''$. Our next goal is to define
a bounded operator $\widetilde{\pi}_\varphi(X)$ acting on $\ch_\varphi$. We start by defining it on the dense subspace
corresponding to the cyclic vector $\xi_\varphi$ in the following way:
\begin{equation}\label{defemb}
\widetilde{\pi}_\varphi(X) \pi_\varphi(\sum_{k\in F} V^k b_k)\xi_\varphi=\pi_\varphi (\sum_{k\in F} V^k\a^k(X)b_k)\xi_{\om_0}
\end{equation}
where $F\subset\bz$ is a  finite set of indices and $b_k$ an element of $\ga$ for all $k$ in $F$.\\
In order to verify that \eqref{defemb} yields a well-defined bounded operator on $\pi_\varphi((\ga\rtimes_\a\bz))\xi_\varphi$, we next show that the inequality
$$ \|\pi_\varphi (\sum_{k\in F} V^k\a^k(X)b_k)\xi_{\om_0}\|^2\leq \|X\|^2\|\pi_\varphi(\sum_{k\in F} V^k b_k)\xi_\varphi\|^2$$
holds. This can be seen by taking a net $\{X_\beta: \beta\in I\}$ (say $X_\beta=\pi_{\om_0}(a_\beta)$ for some
net $\{a_\b: \b\in I\}$) in $\ga$ with $\|X_\beta\|\leq \|X\|$ for all $\beta$ in $I$.
\begin{align*}
&\|\pi_\varphi (\sum_{k\in F} V^k\a^k(X)b_k)\xi_{\om_0}\|^2=\lim_\b\|\pi_\varphi (\sum_{k\in F} V^k\a^k(X_\b)b_k)\xi_{\om_0}\|^2\\
&=\lim_\b\|\pi_\varphi (\sum_{k\in F} V^k\a^k(\pi_{\om_0}(a_\beta))b_k)\xi_{\om_0}\|^2=
\lim_\b \| \pi_{\om_0}(a_\b)\pi_\varphi(\sum_{k\in F}V^kb_k)\xi_\varphi\|^2\\
&\leq\|X\|^2 \|\pi_\varphi(\sum_{k\in F}V^kb_k)\xi_\varphi\|^2\,.
\end{align*}
All is left to do is show $\widetilde{\pi}_\varphi$ is a representation. Now from its definition it is obvious that
the restriction of $\widetilde{\pi}_\varphi$ to $\pi_{\om_0}(\ga)$ is $*$-preserving and multiplicative.
The full conclusion will be reached by strong density of  $\pi_{\om_0}(\ga)$ in its bicommutant  if we prove that 
$\widetilde{\pi}_\varphi$ is a strongly continuous map. Now, as a straightforward  consequence of \eqref{defemb} we see that $\widetilde{\pi}_\varphi$ is strongly continuous on the unit ball of
$\pi_{\om_0}(\ga)''$. Thus by the Kaplansky density theorem the map $\widetilde{\pi}_\varphi$ is strongly continuous on the whole $\pi_{\om_0}(\ga)''$.
\end{proof}

\begin{rem}\label{cov}
For any operator $X$ sitting in the von Neumann algebra $\pi_{\om_0}(\ga)''$ one has
$$\widetilde{\pi}_\varphi (V_g^{\om_0}X(V_g^{\om_0})^*)= V_g^\varphi\widetilde{\pi}_\varphi(X)(V_g^\varphi)^*\,.$$
The above equality follows from  $\widetilde{\pi}_\varphi(\pi_{\om_0}(a))=\pi_\varphi(a)$, $a$ in $\ga$. Indeed, one has $\widetilde{\pi}_\varphi(V_g^{\om_0}\pi_{\om_0}(a)(V_g^{\om_0})^*)=\widetilde{\pi}_\varphi(\pi_{\om_0}(\theta_g(a)))=\varphi(\theta_g(a))=V_g^\varphi \pi_\varphi(a) {V_g^{\varphi }}^*$ for all $a$ in $\ga$, whence the conclusion by
normality of  $\widetilde{\pi}_\varphi$.
\end{rem}
 We are eventually ready to provide the structure of the set of all invariant states.
\begin{thm}\label{invstates}
If a skew-product $(\ga\rtimes_\a\bz, G,  \Phi^u )$  is not uniquely ergodic, then the
 map $\Psi\colon \cp(\bt)\rightarrow \cs(\ga\rtimes_\a\bz)^G$
given by 
$$\Psi(\mu):=\varphi_\mu\circ T,\,\, 	\mu\in \cp(\bt)$$
is an affine homeomorphism between compact convex sets.
\end{thm}
\begin{proof}
Under the assumption that the system is not uniquely ergodic, the smallest natural number, $n_0$, corresponding to which there is a non-trivial solution of the cohomological equation is well defined as is the map $T$ in Lemma
\ref{Tonmon}.\\
The map $\Psi$ is clearly affine. We next show it is a bijection. To this end, we first note that $\Psi$ is injective, as follows from the density of the range of $T$ in $C(\bt)$ established in Lemma \ref{RanT}. Proving its surjectivity is a far more demanding task.  Let  $\eta$ be a  $G$-invariant state on  $\ga\rtimes_\a\bz$. In order to exhibit  the state $\varphi_\mu$ on $C^*(V^{n_0})$ our state $\eta$ comes from, we need to consider a representation
$\rho_\eta$ of $C(\bt)$ acting on the Hilbert space $\ch_\eta$ which is defined as
$$\rho_\eta(z^l):= \widetilde{\pi}_\eta (w_{ln_0})\pi_\eta(V^{ln_0})\,, l\in\bz\,.$$
Let $\mu$ be the spectral measure of the above representation associated with $\xi_\eta$, namely the Borel probability
measure uniquely determined by $\int_\bt f{\rm d}\mu=\langle \rho_\eta(f)\xi_\eta, \xi_\eta \rangle$, for all $f$ in $C(\bt)$.
We claim that $\mu$ is precisely the sought measure, namely that
$\varphi_\mu\circ T=\eta$. To this aim, it is enough to make sure that
$\varphi_\mu(T(aV^k))=\eta(aV^k)$ holds for all $a$ in $\ga$ and $k$ in $\bz$. There are two cases.
Suppose $k$ is not a multiple of $n_0$. In this case, $\eta(aV^k)=0$, as we saw over the course of the proof
of Theorem \ref{uniquelyerg}, and $T(aV^k)=0$ too,  which we saw at the beginning of the proof of Lemma \ref{RanT}.\\
If now $k$ is a multiple of $n_0$, say $k=ln_0$,  we have 
$$\varphi_\mu(T(aV^{ln_0}))=\widetilde{\om}_0(aw_{ln_0}^*)\varphi_\mu(V^{ln_0})\, .$$
On the other hand, we can compute $\eta(aV^{ln_0})$ in the following way:
\begin{align*}
&\eta(aV^{ln_0})=\langle  \pi_\eta(aV^{ln_0})\xi_\eta, \xi_\eta \rangle= \langle\pi_\eta(a)\widetilde{\pi}_\eta(w_{ln_0}^*)\widetilde{\pi}_\eta(w_{ln_0})\pi_\eta(V^{ln_0})\xi_\eta, \xi_\eta \rangle\\
&=\frac{1}{\mu_G(F_n)}\int_{F_n}\langle V_g^\eta\pi_\eta(a)\widetilde{\pi}_\eta(w_{ln_0}^*)(V_g^\eta)^* V_g^\eta \widetilde{\pi}_\eta(w_{ln_0})\pi_\eta(V^{ln_0})\xi_\eta, \xi_\eta \rangle{\rm d}\mu_G(g)
\end{align*}
Now $\frac{1}{\mu_G(F_n)}\int_{F_n} V_g^\eta\pi_\eta(a)\widetilde{\pi}_\eta(w_{ln_0}^*)(V_g^\eta)^*{\rm d}\mu_G(g)$ 
converges as a sequence in $L^\infty(X_0, \mu_0)$  to $\widetilde{\om}_0(aw_{ln_0}^*)$ $\mu_0$ a.e. thanks to Lindenstrauss' pointwise ergodic theorem, \cite[Theorem 1.2]{Lind}.
We claim that $\widetilde{\pi}_\eta(w_{ln_0})\pi_\eta(V^{ln_0})\xi_\eta$  is $V_g^\eta$-invariant, which means $$V_g^\eta \widetilde{\pi}_\eta(w_{ln_0})\pi_\eta(V^{ln_0})\xi_\eta=\widetilde{\pi}_\eta(w_{ln_0})\pi_\eta(V^{ln_0})\xi_\eta\,, g\in G.$$
But then
\begin{align*}
&\lim_{n\rightarrow\infty}\frac{1}{\mu_G(F_n)}\int_{F_n} \langle V_g^\eta\pi_\eta(a)\widetilde{\pi}_\eta(w_{ln_0}^*)(V_g^\eta)^* V_g^\eta\widetilde{\pi}_\eta(w_{ln_0})\pi_\eta(V^{ln_0})\xi_\eta, \xi_\eta \rangle{\rm d}{\mu_G}(g)\\
&=\widetilde{\om}_0(aw_{ln_0}^*)\langle \widetilde{\pi}_\eta(w_{ln_0})\pi_\eta(V^{ln_0})\xi_\eta, \xi_\eta \rangle= \widetilde{\om}_0(aw_{ln_0}^*)\langle \rho_\eta(z^l)\xi_\eta, \xi_\eta \rangle \\
&= \widetilde{\om}_0(aw_{ln_0}^*)\varphi_\mu(V^{ln_0})\,.
\end{align*}
To conclude, we prove the claim. To this aim, we start from the equation satisfied by $w_{ln_0}$ written in the form
$V_g^{\om_0}w_{ln_0}(V_g^{\om_0})^*\pi_{\om_0}(u_g^{(ln_0)})=w_{ln_0}$. Taking $\widetilde{\pi}_\eta$ on both sides of the equation and taking into account Remark \ref{cov}, we find $V_g^\eta\widetilde{\pi}_\eta(w_{ln_0})(V_g^\eta)^* \pi_\varphi(u_g^{(ln_0)})=\widetilde{\pi}_\eta(w_{ln_0})$. But then we have
\begin{align*}
&V_g^\eta\widetilde{\pi}_\eta(w_{ln_0})\pi_\eta(V^{ln_0})\xi_\eta=V_g^\eta\widetilde{\pi}_\eta(w_{ln_0})(V_g^\eta)^*V_g^\eta\pi_\eta(V^{ln_0})\xi_\eta\\
&=\widetilde{\pi}_\eta(w_{ln_0})\pi_\varphi(u_{ln_0}^*)V_g^\eta\pi_\eta(V^{ln_0})(V_g^\eta)^*\xi_\eta\\
&=\widetilde{\pi}_\eta(w_{ln_0})\pi_\varphi((u_g^{(ln_0)})^*)\pi_\eta(u_g^{(ln_0)}V^{ln_0})\xi_\eta\\
&=\widetilde{\pi}_\eta(w_{ln_0})\pi_\eta(V^{ln_0})\xi_\eta\,,
\end{align*}
and the proof is now complete.
\end{proof}

We conclude the section with a result which, among other things, proves that Abadie and Dykema's problem
can be answered in the affirmative in the class of skew-products.

\begin{thm}\label{AbDyk}
Let $(\ga\rtimes_\a\bz, G,  \Phi^u)$ be askew-product system, with $(\ga, G, \theta)$ being a uniquely ergodic system with unique invariant state $\om_0$. The following conditions are equivalent:
\begin{itemize}
\item [(i)] $(\ga\rtimes_\a\bz, G, \Phi^u)$ is uniquely ergodic w.r.t. the fixed-point subalgebra  $(\ga\rtimes_\a\bz)^G$;
\item [(ii)] There exists a unique $G$-invariant conditional expectation from $\ga\rtimes_\a\bz$ onto
$(\ga\rtimes_\a\bz)^G$;
\item [(iii)] for every $n$ in $\bn$, if the cocycle $(u_g^{(n)})_{g\in G}$ is a coboundary, then it is a continuous coboundary;
\end{itemize}
\end{thm}

\begin{proof}
The implication (i)$\Rightarrow$(ii) holds in general. 
As for the implication (ii)$\Rightarrow$(iii), suppose that (iii) does not hold. Then the cohomological
equations also have non-continuous solutions. We denote by $n_0$ the smallest natural number
for which there is a (non-continuous) solution and by $m_0$ the smallest natural number for which the solution
is continuous. We have $m_0=kn_0$ for some (possibly infinite) $k>0$. Following \cite{DFR2}, we next show that under the current hypotheses we can exhibit many different $\Phi_{\th,u}$-invariant conditional expectations.
To this end, we continue to identify $C^*(V^{n_0})$ with $C(\bt)$.
With this identification, we may think of the map $T$ in Lemma \ref{RanT} as taking values in $C(\bt)$, \emph{i.e.}
$T(aV^{ln_0})=\widetilde{\om}_0(aw_{l_0}^*)z^l$. As recalled towards the end of page $13$ of
\cite{DFR2},  there are many different
conditional expectations from $C(\bt)$ onto its subalgebra $C^*(\chi_k)$ generated by the character $\chi_k(z)=z^k$, $z$ in $\bt$. Denote by $F$ any such conditional expectation.\\
Now $C^*(\chi_k)\cong C^*(w_{kn_0}V^{kn_0})=(\ga\rtimes_\a\bz)^ {\Phi_{\th,u}}$  through the $*$-isomorphism $\Psi\colon C^*(\chi_k)\rightarrow C^*(w_{kn_0}V^{kn_0})$ such that $\Psi(\chi_k):=w_{kn_0}V^{kn_0}$.\\
By construction, the composition $E_F=\Psi\circ F\circ T$ provides a $G$-invariant conditional  expectation
onto  $(\ga\rtimes_\a\bz)^G$ and we are done because this procedure yields many different conditional expectations
as explained at length in \cite[Proposition 3.10]{DFR2}.\\
\end{proof}

\appendix
\section{}
We collect here a bunch of possibly known results which are needed to give a self-contained proof of Lemma
\ref{autoconj}, which in turn is one of the key tools to prove Theorem \ref{classification}.

\begin{lem}\label{minimal}
Let $X$ be a compact Hausdorff space with infinitely many points. If $\a$ is a minimal homeomorphism of
$X$, then for any $n\neq 0$ one has  $(f\circ \a^n) g=fg$ for all $f$ in $C(X)$ implies $g=0$.
\end{lem}

\begin{proof}
We shall argue by contradiction. Let $n$ be a fixed non-zero integer.
Suppose there exists $x_0$ in $X$ with $g(x_0)\neq 0$. Then simplifying
the given equality by $g(x_0)$, we find $f(\a^n(x_0))=f(x_0)$. Under our hypotheses, we can assume
$\a^n(x_0)\neq x_0$, for otherwise $\{\a^k(x_0)\colon k\in\bn\}$ would be a proper closed subset invariant under
$\a$. The contradiction is now reached by considering a continuous function $f$ that separates $x_0$ and
$\a^n(x_0)$.
\end{proof}

\begin{rem}\label{measurable}
As a straightforward application of the above result, we also find the following.
In the same hypotheses as  Lemma \ref{minimal}, let  $\mu_0$ be a faithful $\a$-invariant  probability
measure on $X$. If $g$ in $L^\infty(X,\mu_0)$ satisfies  $f(\a^n(x))g(x)=f(x)g(x)$ $\mu_0$ a.e. for all $f$ in $L^\infty(X,\mu_0)$ for some $n\neq 0$, then $g$ is zero $\mu_0$ a.e.
\end{rem}

\begin{lem}\label{MASA}
If $\ga=C(X)$ is a commutative $C^*$-algebra with infinite spectrum $X$ and  $\a$ is a minimal homeomorphism of
$X$, then $\ga$ is maximal abelian in $\ga\rtimes_\a\bz$.
\end{lem}

\begin{proof}
We need to show that if $x$ in  $\ga\rtimes_\a\bz$ commutes with any $a$ in $\ga$, then $x$ itself sits in $\ga$.
Recall that for any such $x$ the Fourier expansion  $x=\sum_n V^n E(V^{-n}x)$ holds w.r.t. the norm topology if the convergence is understood in the Cesaro sense.
In terms of the above expansion, the equality $xa=ax$ reads $\sum_n V^n E(V^{-n}x)a=\sum_n a V^n E(V^{-n}x)$, that is
$$\sum_n V^n E(V^{-n}x)a=\sum_n  V^n \a^{-n}(a) E(V^{-n}x)\, .$$
By uniqueness of the Fourier expansion, we must have $\a^{-n}(a) E(V^{-n}x)=E(V^{-n}x)a$ for all $a$ in $\ga$, which implies
$E(V^{-n}x)=0$ for all $n\neq 0$ thanks to Lemma \ref{minimal}. But then $x=E(x)$ and we are done.
\end{proof}

\begin{lem}
Let $X$ be an infinite compact Hausdorff space and $\a$ a homeomorphism of $X$. 
If $(C(X), \a)$ is a a uniquely ergodic system whose unique invariant state $\om_0$ is faithful
then $\pi_{\om_0}(\ga)''$ is maximal abelian in $\pi_\om(\ga\rtimes_\a\bz)''$.
\end{lem}

\begin{proof}
It is an easy adaptation of the proof of Lemma \ref{MASA} by taking into account
that the Fourier expansion holds at the von Neumann algebra level as well, namely for any
$X$ in $\pi_\om(\ga\rtimes_\a\bz)''$ one has $X=\sum_n \pi_\om(V^n) \widetilde{E}(V^{-n}x)$, see 
\cite  [Proposition 4.2]{DFR1}, and exploiting Remark \ref{measurable}.
\end{proof}

\begin{lem}\label{autoconj}
If $\Psi$ is an automorphism of $\ga\rtimes_\a\bz$ (or $\pi_\om(\ga\rtimes_\a\bz)''$) with
 $\Psi(a)=a$  (or $\Psi(\pi_\om(a))=\pi_\om(a)$) for all $a$ in $\ga$, then there exists a unitary
$w$ in $\ga$ (or $\pi_\om(\ga)''$) such that $\Psi(V)=wV$ (or $\Psi(\pi_\om(V))=w\pi_\om(V)$).
\end{lem}

\begin{proof}
We only deal with the $C^*$-case, for the $W^*$-case can be handled in the same way.
From the equality $VaV^*=\a(a)$, $a\in\ga$, we find $\Psi(V)a\Psi(V^*)=\a(a)$, which means
$VaV^*=\Psi(V)a\Psi(V^*)$, $V^*\Psi(V)a=aV^*\Psi(V)$ for all $a$ in $\ga$. 
The conclusion follows by Lemma \ref{MASA}, which implies that $V^*\Psi(V)$ is a unitary
$w$ sitting in $\ga$.
\end{proof}

\section*{Acknowledgements}

We would like to thank Dan Ursu for allowing us to include in the present paper his solution, Example \ref{Ursu},  to the problem posed by Abadie and Dykema,  which he let us have in a private communication. \\
The authors acknowledge INdAM-GNAMPA Project 2024 ``Probabilita' quantistica e applicazioni''  \\
V.C., S.D.V. and S.R. are also supported by Italian PNRR Partenariato Esteso PE4, NQSTI, and Centro Nazionale CN00000013 CUP H93C22000450007.
M.E.G. is supported by Italian PNRR MUR project PE0000023-NQSTI, CUP H93C22000670006, and  Progetto ERC SEEDS UNIBA ``$C^*$-algebras and von Neumann algebras in Quantum Probability'', CUP H93C23000710001.

\end{document}